\documentclass{amsart}

\usepackage{hyperref}

\newcommand{\ghat}{{\hat{g}}}
\newcommand{\comment}[1]{}
\renewcommand{\mp}{{\lfloor\log_p k\rfloor}}
\newtheorem{theorem}{Theorem}[section]
\newtheorem{lemma}[theorem]{Lemma}

\numberwithin{equation}{section}

\begin{document}

\bibliographystyle{plain}


\title[The Erd\H{o}s-Selfridge Function]{
  An Algorithm and Estimates for the Erd\H{o}s-Selfridge Function \\
  Addendum }

\author{Brianna Sorenson}
\address{Butler University, Indianapolis, IN 46208, USA}
\email{bsorenso@butler.edu}

\author{Jonathan P.~Sorenson}
\address{Butler University, Indianapolis, IN 46208, USA}
\email{sorenson@butler.edu}
\urladdr{blue.butler.edu/$\sim$jsorenso}

\author{Jonathan Webster}
\address{Butler University, Indianapolis, IN 46208, USA}
\email{jewebste@butler.edu}

\maketitle

\section{Introduction}

This is an addendum to our paper of the same title
recently published as part of the proceedings of the ANTS XIV
conference, through Mathematical Sciences Publishers:
\begin{quote}
\url{https://doi.org/10.2140/obs.2020.4.371}
\end{quote}
Due to space and time constraints, we have three minor results
  that did not make the official version of our paper, so
  we present them here.
\begin{enumerate}
\item Two more values of $g(k)$, for $k=376,377$.
\item A proof of the claim at the end of Section 6 that 
  $$
    \limsup_{k\rightarrow\infty} \frac{ \ghat(k+1) }{ \ghat(k) }
    = \infty.
  $$
\item A proof of Lemma 6.4, and hence Theorem 6.1, with an exact
  constant near $0.78843\ldots$.
  The proof in the ANTS paper brackets the constant between 0.5306\ldots and 1.
  Thus, $\log \ghat(k) \sim (0.78843\ldots)\cdot (k/\log k)$.
\end{enumerate}

If you wish to cite this work, we encourage you to cite the ANTS paper
  linked above, and add a note to that citation with a link to this
  arxiv paper if you are specifically referring to the results
  mentioned here.


\section{Two more $g(k)$ values}

We have
\begin{eqnarray*}
  g(376) &=& 7778804220120654420924631668091 \\
  g(377) &=& 5973303871796437264595936954237  \\
\end{eqnarray*}
$g(376)$ took one week, wall time, and
$g(377)$ took about two weeks.


\section{Proof of Claim from the end of Section 6}

\begin{theorem}
  We have
  $$
  \limsup_{k\rightarrow\infty} \frac{ \ghat(k+1) }{ \ghat(k) }
    = \infty.
  $$
\end{theorem}

This proof uses some of the ideas from Section 3 in \cite{LSW97}.

\begin{proof}
  We will prove a lower bound proportional to $\log k$ in the case
    when $k+1$ is an odd prime.
  Since there are infinitely many primes, this will be sufficient to prove the theorem.

  Note that $\ghat(k+1)/\ghat(k) = (M_{k+1}/M_k) (R_k/R_{k+1})$.

  First, we look at $M_{k+1} / M_k$.
  Recall that
  $$
    M_k=\prod_{p\le k} p^{\lfloor \log_p k \rfloor +1}
     \qquad \mbox{and} \qquad
    M_{k+1}=\prod_{p\le {k+1}} p^{\lfloor \log_p (k+1) \rfloor +1}.
  $$
  We can write
  \begin{eqnarray*}
    M_{k+1}
     &=&\prod_{p\le {k+1}} p^{\lfloor \log_p (k+1) \rfloor +1} \\
     &=&(k+1)^2 \cdot \prod_{p\le k} p^{\lfloor \log_p (k+1) \rfloor +1}  \\
     &=&(k+1)^2 \cdot M_k.
  \end{eqnarray*}
  Here we use the fact that for every prime $p\le k$,
     $\lfloor \log_p (k+1) \rfloor = \lfloor \log_p k \rfloor$
    when $k+1$ is prime.

  Next we look at $R_k / R_{k+1}$.
  Using the same notation for $a_{ip}$ as above,
    and noting that the prime $k+1$ will contribute $k(k+1)$ residues, 
    by Kummer's theorem, we have
  \begin{eqnarray*}
     \frac{R_k}{R_{k+1}}
       &=& \frac{
                  \prod_{p\le k} \prod_{i=0}^{\lfloor \log_p k \rfloor}
                       (p-a_{ip})
                }{
                  k(k+1) \cdot
                  \prod_{p\le k} 
                       (p-(a_{0p}+1))
                      \prod_{i=1}^{\lfloor \log_p (k+1) \rfloor}
                       (p-a_{ip})
                } \\
       &=& \frac{1}{k(k+1)}
                  \prod_{p\le k} \frac{ p-a_{0p} }{ p-(a_{0p}+1) }.
  \end{eqnarray*}
  Again we note that
    $\lfloor \log_p (k+1) \rfloor = \lfloor \log_p k \rfloor$,
    and observe that the representation for $k+1$ in base $p$
    is the same as for $k$, with the exception of the least significant
    digit, $a_{0p}$, which is one larger, for all primes $p\le k$.
  This is only because $k+1$ is prime; $k+1 \bmod p$ cannot be zero
    unless $p=k+1$.

  We then bound
     $$\frac{ p-a_{0p} }{ p-(a_{0p}+1) } \ge \frac{p}{p-1}$$
  to obtain that
  $$
     \frac{R_k}{R_{k+1}} \ge \frac{1}{k(k+1)} e^\gamma \log k (1+o(1))
  $$
  using Mertens's theorem.
  We deduce that
  $$
   \frac{ M_{k+1} / R_{k+1} }{ M_k / R_k } \gg \frac{(k+1)^2}{k(k+1)}\log k
   \ge \log k
  $$
  to complete the proof.
\end{proof}


\section{A Constant for Theorem 6.1}

Here is our new proof of Theorem 6.1.
All the substantive changes are in Lemma 6.4.

\begin{theorem}[Theorem 6.1]
$$
\frac{\log \ghat(k)}{k/\log k} \quad\sim\quad 0.7884305\ldots
$$
\end{theorem}

Applying the definitions for $M_k$ and $R_k$ above, we have
\begin{eqnarray*}
  \ghat(k)=\frac{M_k}{R_k} &=&
     \frac{ \prod_{p\le k} p^{\mp+1} }{
            \prod_{p\le k} \prod_{i=0}^{\mp} (p-a_{ip}) }  
  \quad =\quad  \prod_{p\le k} \prod_{i=0}^{\mp} \frac{p}{p-a_{ip}}  \\
    &=& \prod_{p\le\sqrt{k}} \prod_{i=0}^{\mp} \frac{p}{p-a_{ip}}  
    \cdot \prod_{\sqrt{k}<p\le k} \prod_{i=0}^{\mp} \frac{p}{p-a_{ip}}  \\
    &=& \prod_{p\le\sqrt{k}} \prod_{i=0}^{\mp} \frac{p}{p-a_{ip}}  
    \cdot \prod_{\sqrt{k}<p\le k} \frac{p}{p-a_{1p}} \frac{p}{p-a_{0p}}.  \\
\end{eqnarray*}
Here we observed that $\mp+1=2$ when $p>\sqrt{k}$.

We will show that the product on the factor involving $a_{0p}$ is
  exponential in $k/\log k$, and is therefore significant;
and the other two factors,
  the product on primes up to $\sqrt{k}$, and the factor with $a_{1p}$,
  are both only exponential in roughly $\sqrt{k}$.

We bound the first product, on $p\le \sqrt{k}$, with the following lemma.
\begin{lemma}[6.2]
    $$\prod_{p\le\sqrt{k}} \prod_{i=0}^{\mp} \frac{p}{p-a_{ip}}  
    \quad\ll\quad e^{3\sqrt{k}(1+o(1))}. $$
\end{lemma}
\begin{proof}
We note that $a_{ip}\le p-1$, giving
\[ 
  \prod_{p\le\sqrt{k}} \prod_{i=0}^{\mp} \frac{p}{p-a_{ip}}  
  \quad\leq\quad  \prod_{p\le\sqrt{k}}  p^{\mp+1} 
  \quad\leq\quad  \prod_{p\le\sqrt{k}}  p^{3\lfloor\log_p\sqrt{k}\rfloor}. \]
From \cite[Ch. 22]{HW} we have the bound
  $\sum_{p\le x} \lfloor \log_p x \rfloor \log p = x(1+o(1))$.
Exponentiating and substituting $\sqrt{k}$ for $x$ gives the desired result.
\end{proof}

Next, we show that the product involving $a_{1p}$ is small.

\begin{lemma}[6.3]
$$
    \prod_{\sqrt{k}<p\le k} \frac{p}{p-a_{1p}} 
      \quad\ll\quad e^{O(\sqrt{k}\log\log k)}.
$$
\end{lemma}

\newcommand{\sk}{{\lfloor\sqrt{k}\rfloor}}

\begin{proof} 
  We split the product at $2\sqrt{k}$.
  For the lower portion, we have
\begin{eqnarray*}
  \prod_{\sqrt{k}<p\le 2\sqrt{k}} \frac{p}{p-a_{1p}} 
  &\le& (2\sqrt{k})^{\pi(2\sqrt{k})} \quad\ll\quad e^{O(\sqrt{k})}.
\end{eqnarray*}
For the upper portion, we have
\begin{eqnarray*}
  \prod_{2\sqrt{k}<p\le k} \frac{p}{p-a_{1p}} 
    &\le& \prod_{2\sqrt{k}<p\le k} \frac{p}{p-\sqrt{k}}
    \quad = \quad
     \prod_{2\sqrt{k}<p\le k}\left(1+ \frac{2\sqrt{k}}{p} \right) \\
   &\le&
     \prod_{2\sqrt{k}<p\le k}\left(1+ \frac{1}{p} \right)^{2\sqrt{k}+1}
\end{eqnarray*}
using the fact that $(1+x/p) \le (1+1/p)^x$ if $x,p$ are positive integers.
Mertens's Theorem then gives the bound 
$$(e^\gamma (\log k) (1+o(1)))^{2\sqrt{k}+1} \quad\ll\quad
  e^{O( \sqrt{k}\log\log k)}.$$
\end{proof}

We now have
$$
  \log \ghat(k) =
    \log\left(\prod_{\sqrt{k}<p<k} \frac{p}{p-a_{0p}}\right) 
    +O(\sqrt{k}\log\log k).
$$
The following lemma wraps up the proof of our theorem.

\begin{lemma}[6.4 - new]
There exists a constant $c$ where $c\approx 0.7884305\ldots$ where
$$
 \log\left(   \prod_{\sqrt{k}<p\le k} \frac{p}{p-a_{0p}} \right)
  \quad = \quad
   c \cdot \frac{k}{\log k} (1+o(1)).
$$
\end{lemma}

\begin{proof}
Fix $a_{1p}=a$.  
Then $k/(a+1)<p\le k/a$, and $a_{0p}=k\bmod p=k-ap$
  and $p-a_{0p}=p-(k-ap)=(a+1)p-k$.
We have
\begin{eqnarray*}
  \log\left( \prod_{\sqrt{k}<p\le k} \frac{p}{p-a_{0p}} \right) &=&
  \log\left( 
  \prod_{a=1}^{\sqrt{k}}
    \prod_{k/(a+1)<p\le k/a} \frac{p}{(a+1)p-k} \right) \\
 &=&
  \sum_{a=1}^{\sqrt{k}}
    \sum_{k/(a+1)<p\le k/a} \left( \log p  - \log((a+1)p-k) \right). \\
\end{eqnarray*}
We split this sum into three pieces to start with:
\begin{enumerate}
\item The outer sum for $(\log k)^2 \le a \le \sqrt{k}$, and we show it
  is $o(k/\log k)$.
\item The $\log p$ term only, for $a<(\log k)^2$, and show it is
  $k+o(k/\log k)$.
\item  The $-\log((a+1)p-k)$ term, again for $a<(\log k)^2$,
  and show it is $-k+O(k/\log k)$.
\end{enumerate}
For (1), we have
\begin{eqnarray*}
  \sum_{a=(\log k)^2}^{\sqrt{k}}
    \sum_{k/(a+1)<p\le k/a} \left( \log p  - \log((a+1)p-k) \right)
   & \le &
  \sum_{a=(\log k)^2}^{\sqrt{k}}
    \sum_{k/(a+1)<p\le k/a} \log p  \\
   & \le &
    \sum_{\sqrt{k}<p\le k/(\log k)^2} \log p  \\
\end{eqnarray*}
which is $O(k/(\log k)^2)$ using
 $ \sum_{p<x}\log p = x + o(x/\log x)$.
For (2), we have
\begin{eqnarray*}
  \sum_{a=1}^{(\log k)^2}
    \sum_{k/(a+1)<p\le k/a} \log p 
   & = &
    \sum_{k/(\log k)^2<p\le k} \log p 
\end{eqnarray*}
which is $k+o(k/\log k)$.

For (3), we have
\begin{equation}
  -\sum_{a=1}^{(\log k)^2}\sum_{k/(a+1)<p\le k/a} \log({(a+1)p-k}).
\end{equation}

Using the prime number theorem, this is

$$
  \sum_{k/(a+1)<p\le k/a}  \log({(a+1)p-k}) \\
   =  o(k/\log k) +
    \int_{k/(a+1)}^{k/a} \frac{ \log({(a+1)t-k}) }{\log t} dt .
$$
Next, we substitute $u=(a+1)t/k$ so that $t=ku/(a+1)$ and
  $du=((a+1)/k) dt$ giving $dt=(k/(a+1))du$.

\begin{eqnarray*}
    \int_{k/(a+1)}^{k/a} \frac{ \log({(a+1)t-k}) }{\log t} dt 
  &=&
    \frac{k}{a+1} \int_{1}^{1+1/a} \frac{ \log({ku-k}) }{\log (ku/(a+1))} du \\
  &=&
    \frac{k}{a+1} \int_{1}^{1+1/a} 
      \frac{ \log k +\log(u-1) }{\log k + \log (u/(a+1))} du \\
\end{eqnarray*}

Next, for some algebra.  We use the following two identities:
\begin{eqnarray*}
  \frac{A+C}{A+B} &=& 1 + \frac{C-B}{A+B}. \\
  \frac{1}{A+B} &=& \frac{1}{A} - \frac{B}{A(A+B)}. \\
\end{eqnarray*}
Combining these identities gives
\begin{eqnarray*}
  \frac{A+C}{A+B} &=& 1 + \frac{C-B}{A} - \frac{B(C-B)}{A(A+B)}. \\
\end{eqnarray*}
Applying this, gives
\begin{eqnarray*}
   & & \frac{k}{a+1} \int_{1}^{1+1/a} 
      \frac{ \log k +\log(u-1) }{\log k + \log (u/(a+1))} du \\
  &=& \\
  \frac{k}{a+1} \int_{1}^{1+1/a} du &+&
  \frac{k}{a+1} \int_{1}^{1+1/a} \frac{\log(u-1)-\log(u/(a+1))}{\log k} du \\
 &-& 
  \frac{k}{a+1} \int_{1}^{1+1/a} 
    \frac{\log(u/(a+1))(\log(u-1)-\log(u/(a+1)))}{
      (\log k)(\log k+ \log (u/(a+1))) } du.
\end{eqnarray*}
We take each of these three terms in order.

We have
$$
  \frac{k}{a+1} \int_{1}^{1+1/a} du = \frac{k}{a(a+1)}.
$$
Summing over the $a$ values gives the $k$ term promised above.
Yes, the signs work out.

The next term gives our constant.
\begin{eqnarray*}
  & &\frac{k}{a+1} \int_{1}^{1+1/a} \frac{\log(u-1)-\log(u/(a+1))}{\log k} du \\
  &=&
  \frac{k}{(a+1)\log k}
     \int_{1}^{1+1/a} \log\left(\frac{(a+1)(u-1)}{u}\right) du \\
  &=&
  - \frac{k}{(a+1)\log k}  \log(1+1/a).
\end{eqnarray*}
To see this, note that the indefinite integral of
  $ \log( (a+1)(u-1)/u )$ is
  $ (u-1) \log( (a+1)(u-1)/u ) - \log u$.
We then obtain a constant on our $k/\log k$ term of
$$
  \sum_{a=1}^\infty \frac{\log (1+1/a) }{a+1} = 0.7884305\ldots .
$$
With a little algebra, the third term is easily bounded by
a small constant times $k/( (\log k) (\log k/a)$ which, when
  summed over the $a\le (\log k)^2$, gives $O( k \log\log k /(\log k)^2)$
  which is $o(k/\log k)$.

\end{proof}




\bibliography{all}

\end{document}